\documentclass[reqno]{amsart}

\usepackage{amsmath,amssymb,amsthm}
\usepackage{braket,mathtools,bbm}
\usepackage{amscd}

\usepackage[dvipdfmx]{graphicx}
\usepackage[subrefformat=parens]{subcaption}

\usepackage[dvipsnames]{xcolor}

\usepackage{tikz-cd} 

\newtheorem{lma}{Lemma}[section]
\newtheorem{prop}[lma]{Proposition}

\newtheorem{mainthm}{Theorem}

\newtheorem{cor}[lma]{Corollary}

\theoremstyle{definition}

\numberwithin{equation}{section}

\newcommand{\bbN}{\mathbb{N}}

\newcommand{\bbR}{\mathbb{R}}
\newcommand{\bbC}{\mathbb{C}}

\newcommand{\fgfuchs}{\Gamma}
\newcommand{\nsubg}{N}
\newcommand{\disk}{\mathbb{D}}

\newcommand{\hausdim}{\dim_{H}}

\newcommand{\escset}{\mathcal{E}}
\newcommand{\unifescset}{\mathcal{UE}}
\newcommand{\accumset}{\escset^*}
\newcommand{\limitset}{\Lambda}
\newcommand{\leftshift}{\sigma}
\newcommand{\shiftsp}{\Sigma}
\newcommand{\alphabet}{A}
\newcommand{\incmat}{\mathbb{M}}

\newcommand{\geoflow}{\gamma}

\newcommand{\fundpoly}{D}

\newcommand{\diam}{\mathrm{diam}}

\newcommand{\hypbgrp}{\Gamma}

\newcommand{\identity}{1}

\newcommand{\poincareexp}{\delta}

\newcommand{\supp}{\mathrm{supp}}
\newcommand{\hyperbsp}{\mathbb{H}}
\newcommand{\Isom}{\mathrm{Isom}}

\title[Escaping trajectories on free groups]
{A multifractal analysis for escaping trajectories on free groups}
\author{Ziyu Liu}


\begin{document}


\address{School of Arts and Sciences, Shanghai Dianji University, Shanghai, China} 
\email{05564@sdju.edu.cn}

\begin{abstract}
    We consider a free group extension of a subshift of finite type $\leftshift:\shiftsp\rightarrow\shiftsp$, and consider three sets of points in $\shiftsp$ to which the corresponding trajectories on the free group escape to a given point in the Gromov boundary of the free group in three different senses. Under very mild conditions, we provide a common positive lower bound for the $u$-dimensions of these three sets. We also show that the lower bound is equal to the $u$-dimensions of these three sets when the extension and $u$ have certain symmetries. Moreover, we apply our results to geodesic flows on free group covers of Schottky surfaces, and show that there exists a dimension gap between the three sets we considered and the entire escaping set.
\end{abstract}

\maketitle

\section{Introduction and statement of main results}
\subsection{Main results and their motivation from random walks}
    Random walks on non-elementary hyperbolic groups are known to exhibit a transient nature. In this paper, we will confine our discussions to non-cyclic finitely generated free groups $\fgfuchs$. Consider a random walk on $\fgfuchs$ whose generating probability measure is supported by a finite generating set of $\fgfuchs$. For the case where the random walk is symmetric, Kesten's theorem \cite{Kesten} reveals that the spectral radius of this random walk is bound to be strictly less than one. This in turn gives the transience and also the positivity of the speed of the random walk; we refer to \cite[Chapter~II]{woessbook} for details. When the random walk is non-symmetric, recent results obtained by Dougall and Sharp show that the spectral radius is still strictly less than one \cite[Corollary 1.2]{DougallSharp}. Hence, on the one hand, we know that almost every sample path escapes towards a certain point in the Gromov boundary $\partial\fgfuchs$ of $\fgfuchs$. On the other hand, it is shown in \cite[Theorem 1.1]{MT} that the hitting measure of the random walk on $\partial\fgfuchs$ is non-atomic. In other words, for every $x\in\partial\fgfuchs$, the sample paths escaping towards $x$ are negligible in a measure-theoretic sense. Therefore, a natural question arises; what is the Hausdorff dimension of the set of sample paths escaping towards some $x\in\partial\fgfuchs$? This paper aims to provide some answers to this question.

    To make our setup more precise, we work with free group extensions of a subshift of finite type (SFT) in this paper. Let $\leftshift:\shiftsp\rightarrow\shiftsp$ be a topologically transitive SFT given by some alphabet $\alphabet$ and some incidence matrix $\incmat:\alphabet\times\alphabet\rightarrow\set{0,1}$, and $\fgfuchs$ be a non-cyclic, finitely generated free group as before. Let $\shiftsp^*$ denote the set of all $\incmat$-admissible words over $\alphabet$, and $\chi:\shiftsp^*\rightarrow\fgfuchs$ be a surjection satisfying
    \begin{equation*}
        \chi(\omega\omega') =\chi(\omega)\chi(\omega'),\,\forall\omega,\omega'\in\shiftsp^*\text{ with }\omega\omega'\in\shiftsp^*.
    \end{equation*}
    We will call such $\chi$ a (homomorphic) projection. An additional standing assumption on $\chi$ we will make throughout this paper is the following; the preimage $\chi^{-1}(\identity_\fgfuchs)$ of the identity $\identity_\fgfuchs$ of $\fgfuchs$ is transitive, by which we mean that for all $a,a'\in\alphabet$, there is some $\rho\in\chi^{-1}(\identity_\fgfuchs)$ such that $a\rho a'\in\shiftsp^*$. From a viewpoint of the theory of dynamical systems, the projection $\chi$ defines a $\fgfuchs$-extension of $\leftshift:\shiftsp\rightarrow\shiftsp$, which is $T:\shiftsp\times\fgfuchs\rightarrow\shiftsp\times\fgfuchs$ with $T(\xi,g):=(\leftshift(\xi),g\cdot\chi(\xi_0))$ for all $\xi=\xi_0\xi_1\cdots\in\shiftsp$ and all $g\in\fgfuchs$.
    
    Suppose that the $\fgfuchs$ is freely generated by $\set{e_1,\cdots,e_n}\subseteq\fgfuchs$, meaning that $\fgfuchs$ is isomorphic to the free group $\langle e_1,\cdots,e_n|\varnothing\rangle$. We denote by $\partial\hypbgrp$ the Gromov boundary of $\hypbgrp$, whose elements can be identified with reduced infinite sequences over $\fgfuchs_0:=\set{e_1,e_1^{-1},\cdots,e_n,e_n^{-1}}$; see Subsection \ref{subsec:freegroups} for details. Given $x\in\partial\fgfuchs$, we define
    \begin{align*}
        \accumset_x&:=\Set{\xi\in\shiftsp|\limsup_{m\rightarrow+\infty}|\chi(\xi_0\cdots\xi_{m-1})\wedge x|=+\infty};\\
        \escset_x&:=\Set{\xi\in\shiftsp|\lim_{m\rightarrow+\infty}|\chi(\xi_0\cdots\xi_{m-1})\wedge x|=+\infty};\\
        \unifescset_x&:=\Set{\xi\in\escset_x|\sup_{m\geq 1}\left\{|\chi(\xi_0\cdots\xi_{m-1})|-|\chi(\xi_0\cdots\xi_{m-1})\wedge x|\right\}<+\infty},
    \end{align*}
    where $g\wedge x$ is the longest common prefix of $g\in\fgfuchs$ and $x\in\partial\fgfuchs$. Clearly, we have $\unifescset_x\subseteq \escset_x\subseteq \accumset_x$. In words, $\accumset_x$ consists of the sequences $\xi\in\shiftsp$ for which $\set{\chi(\xi_0\cdots\xi_{m-1})|m\geq 1}$ has $x$ as an accumulation point; $\escset_x$ consists of the sequences $\xi\in\shiftsp$ for which the trajectory $\chi(\xi_0\cdots\xi_{m-1})$ converges towards $x$ as $m\rightarrow+\infty$; $\unifescset_x$ consists of the sequences $\xi\in\shiftsp$ for which the trajectory $\chi(\xi_0\cdots\xi_{m-1})$ converges to $x$ as $m\rightarrow+\infty$ and lies within a bounded distance from the geodesic joining the identity $\identity_\fgfuchs\in\fgfuchs$ and $x$. We will estimate the Hausdorff dimension of these sets. For this purpose, let $u:\shiftsp\rightarrow(0,+\infty)$ be a H\"older continuous function, and we endow $\shiftsp$ with the metric $d_u$; see Subsection \ref{subsec:pre:symbdyn} for details. It can be checked \cite[Proposition 3.3]{liuthesis} that the Hausdorff dimension for the metric $d_u$ coincides with the (Barreira-Schmeling) $u$-dimension in \cite{BarSch00}. Hence, given any $E\subseteq\shiftsp$, we will only estimate the $u$-dimension $\dim_u(E)$ of $E$.
    
    Our first main theorem, namely Theorem \ref{mainthm:dimspec:general} below, gives a lower bound for the $u$-dimensions of $\accumset_x$, $\escset_x$ and $\unifescset_x$ in terms of a critical exponent $\poincareexp(\identity_\fgfuchs)$ which we shall call the Poincar\'e exponent restricted to $\chi^{-1}(\identity_\fgfuchs)$.
    
    \begin{mainthm} \label{mainthm:dimspec:general}
        Let $\leftshift:\shiftsp\rightarrow\shiftsp$ be a topologically transitive SFT, and $\fgfuchs$ be a non-cyclic, finitely generated free group. Suppose that $\chi:\shiftsp^*\rightarrow\fgfuchs$ is a projection with $\chi^{-1}(\identity_\fgfuchs)$ being transitive. Let $u:\shiftsp\rightarrow(0,+\infty)$ be a H\"older continuous function. Then, for all $x\in\partial\fgfuchs$, we have
        \begin{equation} \label{eqn:dimspec:ineq}
            \dim_u(\accumset_x)\geq \dim_u(\escset_x) \geq\dim_u(\unifescset_x)\geq\poincareexp(\identity_\fgfuchs)>0,
        \end{equation}
        where $\poincareexp(\identity_\fgfuchs):=\inf\set{p\geq 0|\sum_{\omega\in\chi^{-1}(\identity_\fgfuchs)}\exp(-pS_\omega u)<+\infty}$.
    \end{mainthm}

    Our next main theorem, namely Theorem \ref{mainthm:dimspec:symmetric} below, states that all the inequalities in \eqref{eqn:dimspec:ineq} are in fact equalities, assuming the symmetries of the $\fgfuchs$-extension and $u$ in a sense described in \cite{Stad13}. The precise definitions of these symmetries will be given at the end of Subsection \ref{subsec:proj}.
    \begin{mainthm} \label{mainthm:dimspec:symmetric}
        Assume all the conditions in Theorem \ref{mainthm:dimspec:general}. Let $\cdot^\dagger:\shiftsp^*\rightarrow\shiftsp^*$ be an involution, and assume the symmetry of the $\fgfuchs$-extension given by the projection $\chi$ with respect to $\cdot^\dagger$ and also the symmetry of $u$ with respect to $\cdot^\dagger$. Then, we have
        \begin{equation} \label{eqn:dimspec:eq}
            \dim_u(\accumset_x)= \dim_u(\escset_x) =\dim_u(\unifescset_x)=\poincareexp(\identity_\fgfuchs).
        \end{equation}
    \end{mainthm}

    We give a rough sketch of the two proofs of Theorem \ref{mainthm:dimspec:general} and Theorem \ref{mainthm:dimspec:symmetric}. For each $g\in\fgfuchs$ and $p\in\bbR$, we define what we call the Poincar\'e series restricted to $\chi^{-1}(g)$ as
    \begin{equation*}
        Z(p|g):=\sum_{\omega\in\chi^{-1}(g)}\exp(-pS_\omega u).
    \end{equation*}
    Clearly $\poincareexp(\identity_\fgfuchs)$ is the critical exponent for convergence of $Z(p|\identity_\fgfuchs)$. The lower bound for $\dim_u(\unifescset_x)$ in Theorem \ref{mainthm:dimspec:general} will be derived by applying the mass distribution principle to a Borel probability measure whose construction relies on the divergence of $Z(p|\identity_\fgfuchs)$. This argument is inspired by the proofs in some recent papers \cite{gjk22,liu}. The proof of the upper bound for $\dim_u(\accumset_x)$ in Theorem~\ref{mainthm:dimspec:symmetric} will be shown to amount to the proof of
    \begin{equation*}
        \sum_{m=1}^\infty Z(p|x_0\cdots x_m)<+\infty.
    \end{equation*}
    This is achieved in the following way. Firstly, we will see in Proposition \ref{prop:decrateZ} that it always holds, even without the assumptions on symmetries, that
    \begin{equation*}
        \sum_{m=1}^\infty \frac{Z(p|x_0\cdots x_m)}{Z(p'|x_0\cdots x_m)}<+\infty
    \end{equation*}
    if $\poincareexp(\identity_\fgfuchs)<p'<p$. Assuming the conditions on symmetries and using the almost supermultiplicativity of restricted Poincar\'e series we are to show in Proposition~\ref{prop:poincareseries:supadd}, we will see that $\sup_{m\geq 1}Z(p'|x_0\cdots x_m)<+\infty$ for all $p'\in(\poincareexp(\identity_\fgfuchs),p)$. From the previous two facts, we have the convergence of $\sum_{m=1}^\infty Z(p|x_0\cdots x_m)$.

    Now we continue our discussions about random walks on free groups we made at the beginning. Let $\fgfuchs$ be a non-cyclic free group generated by some balanced finite set $\alphabet\subseteq \fgfuchs$; we say $E\subseteq \fgfuchs$ is balanced if $E^{-1}:=\set{g^{-1}|g\in E}$ equals $E$. Consider the random walk on $\fgfuchs$ generated by a probability measure $\mu_0$ whose support $\supp(\mu_0)$ is exactly $\alphabet$. The set of sample paths can thus be described in terms of a $\fgfuchs$-extension of an SFT. Let $\leftshift:\shiftsp\rightarrow\shiftsp$ be the full shift given by the alphabet $\alphabet$, and define $\chi(\omega)$ as the product $\omega_0\cdots\omega_{|\omega|-1}\in\fgfuchs$ for each $\omega=\omega_0\cdots\omega_{|\omega|-1}\in\shiftsp^*=\bigcup_{k=0}^\infty\alphabet^k$. Then, $\chi$ naturally gives a one-to-one correspondence between the elements of $\shiftsp$ and the sample paths starting from $\identity_\fgfuchs$, by sending every $\xi=\xi_0\xi_1\cdots\in\shiftsp$ to the sample path $(\chi(\xi_0\cdots\xi_{m-1}))_{m\in\bbN}$. As is remarked at the beginning of this paper, for any $x\in\partial\fgfuchs$, the set $\mathcal{R}:=\set{\xi\in\shiftsp|\liminf_{m\rightarrow+\infty}|\chi(\xi_0\cdots\xi_{m-1})|<+\infty}$ of recurrent sample paths and the set $\escset_x$ are both null sets, showing that
    \begin{equation*}
        \unifescset_x\subseteq\escset_x\subseteq\accumset_x\subseteq\mathcal{R}\cup\escset_x
    \end{equation*}
    are all negligible from the measure-theoretical perspective. On the other hand, our Theorem \ref{mainthm:dimspec:general} provides a \emph{positive} lower bound for the $u$-dimensions of all these sets. To see that Theorem \ref{mainthm:dimspec:general} can be applied, we note that $\chi$ is clearly a projection and $\chi^{-1}(\identity_\fgfuchs)$ is transitive because $\leftshift:\shiftsp\rightarrow\shiftsp$ is a full shift.

\subsection{Application to geodesic flows on free group covers of Schottky surfaces}
    Our results can also be applied to the geodesic flows on free group covers of Schottky surfaces. Let $\hyperbsp$ be the open unit disk in $\bbC$ endowed with the hyperbolic metric $d_h$. Let $G$ be a Schottky group generated by the set $\alphabet$ of side-pairing transformations; see Subsection \ref{subsec:schottky:def} for the definition. Then $G$ is known to be a finitely generated free group, and we assume that the rank of the free group $G$ is at least three. We call the quotient surface $\hyperbsp/G$ a Schottky surface. Let $\fgfuchs$ be a non-cyclic free group of rank lower than the rank of $G$, and $\chi:G\rightarrow \fgfuchs$ be a surjective group homomorphism. We are interested in the infinitely generated free group $\nsubg:=\ker(\chi)=\chi^{-1}(\identity_\fgfuchs)$, and the quotient surface $\hyperbsp/\nsubg$, which is a $\fgfuchs$-cover of the Schottky surface $\hyperbsp/G$. For every $z\in\hyperbsp$, we will write $\nsubg z$ to denote the image of $z$ under the natural projection mapping $\hyperbsp$ onto $\hyperbsp/\nsubg$.
    
    We shall investigate the geodesic flow on $\hyperbsp/\nsubg$. More precisely, we fix a point $\nsubg 0\in\hyperbsp/\nsubg$, and consider the geodesic rays which starts from $\nsubg 0$ but have different initial directions. In order to describe these geodesic rays, we can lift them to the universal covering space $\hyperbsp$. For every $\xi$ on the unit circle $\partial\hyperbsp$ and every $t\geq 0$, we let $\geoflow_\xi(t)$ be the unique point on the line segment joining $0$ and $\xi$ such that $d_h(0,\geoflow_\xi(t))=t$. Then every geodesic ray on $\hyperbsp/\nsubg$, parameterized by the length and starting from $\nsubg 0$, is $\nsubg \geoflow_\xi:[0,+\infty)\rightarrow\hyperbsp/\nsubg;\,t\mapsto\nsubg\geoflow_\xi(t)$ for some $\xi\in\partial\hyperbsp$.
    
    Let $\limitset(\nsubg)$ be the limit set of $\nsubg$. For $\xi\in\partial\hyperbsp\setminus\limitset(\nsubg)$, then $\nsubg\geoflow_\xi(t)$ eventually drops into a certain funnel of $\hyperbsp/\nsubg$. Hence it is more interesting to consider geodesic rays $\nsubg\geoflow_\xi$ for $\xi\in\limitset(\nsubg)$. Among the limit points, the radial limit points and uniform radial limit points are well studied and understood. If $\nsubg\geoflow_\xi(t)$ is recurrent to some compact subset of $\hyperbsp/\nsubg$, then $\xi$ is said to be a radial limit point $\nsubg$. If $\nsubg\geoflow_\xi(t)$ always stays in some compact subset of $\hyperbsp/\nsubg$ for all $t\geq 0$, then $\xi$ is said to be a uniform radial limit point of $\nsubg$. It has been long known that the radial limit set $\limitset_r(\nsubg)$ of $\nsubg$, consisting of the radial limit points of $\nsubg$, has Hausdorff dimension equal to the Poincar\'e exponent of $\nsubg$ \cite{radialdim}, defined as
    \begin{equation*}
        \poincareexp(\nsubg):=\inf\Set{s\geq 0|\sum_{h\in \nsubg}\exp(-sd(0,h0))<+\infty},
    \end{equation*}
    and the uniform radial limit set $\limitset_{ur}(\nsubg)$ of $\nsubg$, consisting of the uniform radial limit points of $\nsubg$, has the same Hausdorff dimension \cite{uradialdim}. On the other hand, since $\nsubg$ is a non-elementary normal subgroup of $G$, the limit set $\limitset(\nsubg)$ of $\nsubg$ is precisely the limit set $\limitset(G)$ of $G$. This implies that $\hausdim(\limitset(\nsubg))=\hausdim(\limitset(G))$, which is further equal to the Poincar\'e exponent $\poincareexp(G)$ of $G$ because $\limitset(G)=\limitset_r(G)$ for the Schottky group $G$ which is convex-cocompact. Now that $G/\nsubg$ is a non-cyclic free group, which is non-amenable, \cite{Stad13} shows that $\poincareexp(G)>\poincareexp(N)$. Summarizing the known facts we mentioned in this paragraph, we have
    \begin{equation} \label{eqn:dimgap:radial}
        \hausdim(\limitset_r(\nsubg)) =\hausdim(\limitset_{ur}(\nsubg))=\poincareexp(\nsubg)<\poincareexp(G)=\hausdim(\nsubg).
    \end{equation}
    
    Now we turn our attention to the limit points $\xi$ for which $\nsubg\geoflow_\xi(t)$ shows a transient behavior, and we will see that a dimension gap similar to \eqref{eqn:dimgap:radial} also occurs. In Subsection \ref{subsec:schottky:def}, we will describe a natural one-to-one correspondence between the reduced sequences over $\alphabet$ and the limit points of $\nsubg$, which are precisely the limit points of the Schottky group $G$. Hence for every $x\in\partial\fgfuchs$, the sets $\accumset_x$, $\escset_x$ and $\unifescset_x$ naturally become
    \begin{align*}
        \accumset_x(\nsubg) &:=\Set{\xi\in\limitset(\nsubg)|\limsup_{m\rightarrow+\infty}|\chi(\xi_0\cdots\xi_{m-1})\wedge x|=+\infty};\\
        \escset_x(\nsubg) &:=\Set{\xi\in\limitset(\nsubg)|\lim_{m\rightarrow+\infty}|\chi(\xi_0\cdots\xi_{m-1})\wedge x|=+\infty};\\
        \unifescset_x(\nsubg) &:=\Set{\xi\in\escset_x(\nsubg)|\sup_{m\geq 1}\left\{|\chi(\xi_0\cdots\xi_{m-1})|-|\chi(\xi_0\cdots\xi_{m-1})\wedge x|\right\}<+\infty}.
    \end{align*}
    In addition, we also introduce the \emph{entire escaping set}, defined as
    \begin{equation*}
        \escset(\nsubg) :=\bigcup_{x\in\partial\fgfuchs}\escset_x(\nsubg)=\Set{\xi\in\limitset(\nsubg)|\lim_{m\rightarrow+\infty}|\chi(\xi_0\cdots\xi_{m-1})|=+\infty}.
    \end{equation*}
    Since the Schottky group $G$ is convex-cocompact, there is some compact subset $K$ of a fundamental domain $\fundpoly$ of $G$ such that $\geoflow_\xi(t)$ stays in $\bigcup_{g\in G}gK$ for all $t\geq 0$, provided that $\xi\in\limitset(G)=\limitset(\nsubg)$ \cite[Section II.1]{dalbo}. Hence for every $t\geq 0$ and $\xi\in\limitset(\nsubg)$, $\geoflow_\xi(t)$ is in some $\xi_0\cdots\xi_{m(t)-1}K$, and we have
    \begin{equation*}
        d_h(\nsubg\geoflow_\xi(t),\nsubg\xi_0\cdots\xi_{m(t)-1}0) \leq d_h(\geoflow_\xi(t),\xi_0\cdots\xi_{m(t)-1}0)\leq \diam(K)<+\infty,
    \end{equation*}
    where $d_h(\nsubg\geoflow_\xi(t),\nsubg \xi_0\cdots\xi_{m(t)-1}0)$ is the distance between the points $\nsubg\geoflow_\xi(t)$ and $\nsubg\xi_0\cdots\xi_{m(t)-1}0$ on the quotient surface $\disk/\nsubg$. This shows that the four sets $\accumset_x(\nsubg)$, $\escset_x(\nsubg)$, $\unifescset_x(\nsubg)$ and $\escset(\nsubg)$ actually describe transient behaviors of $\nsubg\geoflow_\xi$, although the geodesic flow does not show up in their definitions. We claim that the following dimension gap exists.

    \begin{mainthm} \label{mainthm:geodflow}
        Let $G$ be a Schottky group, $\fgfuchs$ be a non-cyclic free group and $\chi:G\rightarrow\fgfuchs$ be a surjective group homomorphism. Let $N\triangleleft G$ be the kernel of $\chi$. Then, for all $x\in\partial\fgfuchs$, we have
        \begin{align*}
            \hausdim(\accumset_x(\nsubg))= \hausdim(\escset_x(\nsubg)) &=\hausdim(\unifescset_x(\nsubg))\\&=\poincareexp(\nsubg)<\poincareexp(G)=\hausdim(\limitset(\nsubg))=\hausdim(\escset(\nsubg)).
        \end{align*}
    \end{mainthm}

    The first three equalities are actually a consequence of our Theorem \ref{mainthm:dimspec:symmetric}; the proof will be given in Subsection \ref{subsec:appl:proof}. The only inequality and the last second equality have been explained in the previous paragraph. To see the last equality, we note that $\limitset(\nsubg)=\limitset_r(\nsubg)\cup\escset(\nsubg)$, implying that
    \begin{equation*}
        \hausdim(\limitset(\nsubg)) =\max\set{\hausdim(\limitset_r(\nsubg)),\hausdim(\escset(\nsubg))}.
    \end{equation*}
    From \eqref{eqn:dimgap:radial}, we have $\hausdim(\limitset(\nsubg))>\hausdim(\limitset_r(\nsubg))$, so we obtain $\hausdim(\limitset(\nsubg))=\hausdim(\escset(\nsubg))$. In fact, applying Bishop's result \cite{bishop} to our case, we can say even more; the set of $\xi\in\limitset(\nsubg)$ for which $\nsubg\geoflow_\xi(t)$ escapes with a positive speed actually also has Hausdorff dimension equal to $\hausdim(\limitset(\nsubg))$.

    Lastly, we make a remark on some estimates about $\poincareexp(\nsubg)$. As we will point out in Section \ref{sec:appl}, the dynamics of $\limitset(G)$ given by the Bowen-Series map is conjugate to, and thus can be viewed as a topologically mixing SFT. Then the group homomorphism $\chi$ is in fact a homomorphic projection, and $\poincareexp(\nsubg)$ can be shown to be exactly $\poincareexp(\identity_\fgfuchs)$. Hence, by Theorem \ref{mainthm:dimspec:general}, we have the positivity of $\poincareexp(\nsubg)$. As a matter of fact, $\poincareexp(\nsubg)$ has a much better lower bound. Theorem 2 in \cite{FS} shows that $\poincareexp(\nsubg)\geq \poincareexp(G)/2$. Moreover, since the Schottky group $G$ is of divergence type, we actually further have $\poincareexp(\nsubg)>\poincareexp(G)/2$; see \cite[Theorem 1.1]{jaerisch}.

\subsection{Plan of this paper}
    The rest of this paper is organized in the following way. Section \ref{sec:pre} defines the notions we need to prove the main theorems, Theorem \ref{mainthm:dimspec:general} and Theorem \ref{mainthm:dimspec:symmetric}. In Section~\ref{sec:rpoincareseries} we explore some crucial properties of restricted Poincar\'e series, which will play important roles in the proofs of our main theorems. Section \ref{sec:mainproof} gives the proofs of Theorem \ref{mainthm:dimspec:general} and Theorem \ref{mainthm:dimspec:symmetric}. Section \ref{sec:appl} provides details about the application of Theorem \ref{mainthm:dimspec:symmetric} to geodesic flows on free group covers of Schottky surfaces. In Appendix \ref{app:lemma:proof}, we prove a technical lemma used in the proof of Theorem \ref{mainthm:dimspec:general}.
    
    \emph{Acknowledgments.} \textcolor{red}{TBA.}

\section{Preliminaries} \label{sec:pre}

\subsection{Symbolic dynamics} \label{subsec:pre:symbdyn}
    Let $\alphabet$ be a finite set whose cardinality is at least two, and let $\incmat:\alphabet\times\alphabet\rightarrow\set{0,1}$ satisfy that for all $a\in\alphabet$, there is some $a'\in\alphabet$ such that $\incmat(a,a')=1$. Then, the set
    \begin{equation*}
        \shiftsp :=\Set{\xi=\xi_0\xi_1\cdots\in\alphabet^\bbN|\incmat(\xi_j,\xi_{j+1})=1,\,\forall j\in\bbN},
    \end{equation*}
    which is called the shift space, is non-empty, and define $\leftshift:\shiftsp\rightarrow\shiftsp$ as the left shift, which sends $\xi=\xi_0\xi_1\cdots\in\shiftsp$ to $\leftshift(\xi)=\xi_1\xi_2\cdots\in\shiftsp$. Clearly, $\leftshift:\shiftsp\rightarrow\shiftsp$ is a dynamical system, and is called a subshift of finite type (SFT) given by the alphabet $\alphabet$ and the incidence matrix $\incmat$.

    Given a word $\omega=\omega_0\cdots\omega_{|\omega|-1}$ over $\alphabet$ whose length is $|\omega|$, we define its cylinder set as
    \begin{equation*}
        [\omega] :=\set{\xi\in\shiftsp|\xi_0\cdots\xi_{|\omega|-1}=\omega}\subseteq\shiftsp.
    \end{equation*}
    We say that a word over $\alphabet$ is $\incmat$-admissible if its cylinder set is non-empty, and denote by $\shiftsp^*$ the set of all $\incmat$-admissible words over $\alphabet$. For convenience, we define the cylinder set of the empty word $\epsilon$ to be $\shiftsp$, from which it follows that $\epsilon\in\shiftsp^*$. We say that $\omega\in\shiftsp^*\cup\shiftsp$ is a prefix of $\omega'\in\shiftsp^*\cup\shiftsp$ if $|\omega|\leq |\omega'|$ and $\omega=\omega'_0\cdots\omega'_{|\omega|-1}$, where we think of the length of an element of $\shiftsp$ to be $+\infty$. We define the length $|\epsilon|$ of the empty word to be zero, and regard $\epsilon$ as a common prefix of all the words in $\shiftsp^*$. The longest common prefix of $\omega\in\shiftsp^*\cup\shiftsp$ and $\omega'\in\shiftsp^*\cup\shiftsp$ will be denoted by $\omega\wedge\omega'$.

    Given a function $\phi:\shiftsp\rightarrow\bbR$ and $n\in\bbN$, we define the $n$-th Birkhoff sum of $f$ to be the function $S_n\phi:= \sum_{k=0}^{n-1}\phi\circ\leftshift^k$ defined on $\shiftsp$. We say that a function $\phi:\shiftsp\rightarrow\bbR$ is H\"older continuous if \begin{equation*}
        \sup_{\substack{\xi,\eta\in\shiftsp\\\xi\neq \eta}} \frac{|\phi(\xi)-\phi(\eta)|}{\exp(-\alpha_u|\xi\wedge\eta|)}<+\infty.
    \end{equation*}
    It is widely known that the distortion constant of a H\"older continuous function $\phi:\shiftsp\rightarrow\bbR$, defined by        
    \begin{equation*}
        V_\phi :=\sup_{\omega\in\shiftsp^*}\sup_{\xi,\eta\in[\omega]}|S_{|\omega|}\phi(\xi)-S_{|\omega|}\phi(\eta)|,
    \end{equation*}
    is finite; see for example \cite[Lemma 2.3.1]{MauUrb03}. This property is called the bounded distortion property of the H\"older continuous function $\phi$. Given any H\"older continuous function $\phi:\shiftsp\rightarrow\bbR$, we define
    \begin{equation*}
        S_\omega \phi:=\sup_{\xi\in[\omega]}S_{|\omega|}\phi(\xi),
    \end{equation*}
    for every $\omega\in\shiftsp^*$. The bounded distortion property of $\phi$ then guarantees that for all $\omega\in\shiftsp^*$ and all $\xi\in[\omega]$, we have $S_\omega \phi-V_\phi\leq S_{|\omega|}\phi(\xi)\leq S_\omega\phi$.
    
    In order to endow $\shiftsp$ with a metric, we take a H\"older continuous function $u:\shiftsp\rightarrow(0,+\infty)$, and define a metric $d_u$ of $\shiftsp$ by setting
    \begin{equation*}
        d_u(\xi,\eta) :=\exp(S_{\xi\wedge\eta}(-u))
    \end{equation*}
    for any two distinct $\xi,\eta\in\shiftsp$. It is not hard to check that $d_u$ is an ultrametric, and thus a metric \cite[Proposition 2.4]{liuthesis}. Moreover, for H\"older continuous $u:\shiftsp\rightarrow(0,+\infty)$, the Hausdorff dimension given by $d_u$ is exactly the $u$-dimension in \cite{BarSch00}. The author's thesis contains a proof of this fact \cite[Proposition~3.3]{liuthesis}.

    We will always assume that $\leftshift:\shiftsp\rightarrow\shiftsp$ is topologically transitive, which is equivalent to saying that for all $a,a'\in\alphabet$, there is some $\omega\in\shiftsp^*$ such that the concatenation $a\omega a'$ is also $\incmat$-admissible. Given a topologically transitive SFT $\leftshift:\shiftsp\rightarrow\shiftsp$, we will say that $E\subseteq\shiftsp^*$ is \emph{transitive} if for any $a,a'\in\alphabet$, there is some $\omega\in\shiftsp^*$ such that $a\omega a'\in\shiftsp^*$. Since $\alphabet$ is finite, every transitive $E\subseteq\shiftsp^*$ admits the existence of a finite transitive subset of $E$.

\subsection{Coding of Gromov boundary of free groups} \label{subsec:freegroups}
    Let $\fgfuchs$ be a finitely generated free group of rank $n$, freely generated by $\set{e_1,\cdots,e_n}\subseteq\fgfuchs$. Set
    \begin{equation*}
        \fgfuchs_0 :=\set{e_1,e_1^{-1},\cdots,e_n,e_n^{-1}},
    \end{equation*}
    and we call it a balanced free generating set. Let $\fgfuchs_0$ be the alphabet. Define the incidence matrix $\incmat_\fgfuchs:\fgfuchs_0\times\fgfuchs_0\rightarrow\set{0,1}$ by letting $\incmat_\fgfuchs(e,e')=0$ if and only if $e'=e^{-1}$ for all $e,e'\in\fgfuchs_0$. In this case, we call $\incmat_\fgfuchs$-admissible words reduced words. Then, it is an elementary fact that every group element of $\fgfuchs$ can be uniquely written as a reduced word over $\fgfuchs_0$. In particular, the identity $\identity_\fgfuchs$ of $\fgfuchs$ is represented by the empty word. Henceforth, we will not distinguish group elements of $\fgfuchs$ with the reduced words representing them, so $\fgfuchs$ is identified with the set of all reduced words over $\fgfuchs_0$.

    The Gromov boundary $\partial\fgfuchs$ of the finitely generated free group $\fgfuchs$ is also easy to describe using the language from the symbolic dynamics. Indeed, every element of $\partial\fgfuchs$ can be uniquely represented by a reduced infinite sequence over $\fgfuchs_0$, where we call an infinite sequence over $\fgfuchs_0$ a reduced sequence if it is $\incmat_\fgfuchs$-admissible, as we did for words of finite length. Hence, we regard $\partial\fgfuchs$ as the shift space of the SFT given by the alphabet $\fgfuchs_0$ and the incidence matrix $\incmat_\fgfuchs$.

    Now that we have identified $\fgfuchs$ with the set of reduced words over $\fgfuchs_0$ and $\partial\fgfuchs$ with the shift space of the SFT given by the alphabet $\fgfuchs_0$ and the incidence matrix $\incmat_\fgfuchs$, we will freely use the notations from symbolic dynamics introduced in the previous subsection when we make discussions about $\fgfuchs$ and $\partial\fgfuchs$.

\subsection{Homomorphic projections and group extensions of SFTs} \label{subsec:proj}
    Let $\leftshift:\shiftsp\rightarrow\shiftsp$ be a topologically transitive SFT given by the alphabet $\alphabet$ and the incidence matrix $\incmat:\alphabet\times\alphabet\rightarrow\set{0,1}$. Denote by $\shiftsp^*$ the set of all $\incmat$-admissible words over $\alphabet$. Given a finitely generated free group $\fgfuchs$, a surjection $\chi:\shiftsp^*\rightarrow\fgfuchs$ is said to be a \emph{(homomorphic) projection} if $\chi(\omega\omega') =\chi(\omega)\chi(\omega')$ for any $\omega,\omega'\in\shiftsp^*$ satisfying $\omega\omega'\in\shiftsp^*$. In what follows, the projection $\chi:\shiftsp^*\rightarrow\fgfuchs$ will always assumed to satisfy that $\chi^{-1}(\identity_\fgfuchs)$ is transitive. In this case, the following claim holds.
    \begin{prop} \label{prop:unifirred}
        Suppose that $\chi^{-1}(\identity_\fgfuchs)$ is transitive. Then, for any $r\geq 0$, there exists some $L(r)<+\infty$ such that for all $g,g'\in\fgfuchs$ with $|g^{-1}g'|\leq r$ and all $\omega'\in\chi^{-1}(g')$, there is some $\omega\in\chi^{-1}(g)$ for which $\omega'$ is a prefix of $\omega$ and $|\omega|-|\omega'|\leq L(r)$.
    \end{prop}
    \begin{proof}
        Since $\chi^{-1}(\identity_\fgfuchs)$ is transitive and $\alphabet$ is finite, there is a finite transitive subset $\mathcal{I}$ of $\chi^{-1}(\identity_\fgfuchs)$. Also for $r\geq 0$ and $h\in\fgfuchs$ with $|h|\leq r$, take $\tau(h)\in\chi^{-1}(h)$. Then, given any $g,g'\in\fgfuchs$ with $|g^{-1}g'|\leq r$ and any $\omega'\in\chi^{-1}(g')$, there is some $\rho\in\mathcal{I}$ such that $\omega:=\omega'\rho\tau((g')^{-1}g)\in\shiftsp^*$. Clearly, $\omega'$ is a prefix of $\omega$. Also note that
        \begin{equation*}
            \chi(\omega) =\chi(\omega')\chi(\rho)\chi(\tau((g')^{-1}g)) =g'\cdot\identity_\fgfuchs\cdot (g')^{-1}g=g,
        \end{equation*}
        indicating that $\omega\in\chi^{-1}(g)$. Lastly, observe that
        \begin{equation*}
            |\omega|-|\omega'|=|\rho|+|\tau((g')^{-1}g)|\leq L(r):=\max_{\rho'\in\mathcal{I}}|\rho'|+\max\set{\left|\tau(h)\right| | h\in\fgfuchs,\,|h|\leq r}.
        \end{equation*}
        Since there are finitely many $h\in\fgfuchs$ such that $|h|\leq r$, we have $L(r)<+\infty$, which concludes the proof.
    \end{proof}

    Lastly we define the symmetry of the $\fgfuchs$-extended systems and the symmetry of the potential $u$ on $\shiftsp$. Let $\cdot^\dagger:\shiftsp^*\rightarrow\shiftsp^*$ be an involution, which means that $(\omega^\dagger)^\dagger=\omega$ for all $\omega\in\shiftsp^*$. We say that the $\fgfuchs$-extension of the SFT $\leftshift:\shiftsp\rightarrow\shiftsp$ given by the projection $\chi$ is \emph{symmetric} with respect to $\cdot^\dagger$ if
    \begin{enumerate}
        \item   for any $\omega,\tau\in\shiftsp^*$ satisfying $\omega\tau\in\shiftsp^*$, we have $\tau^\dagger\omega^\dagger\in\shiftsp^*$ and $(\omega\tau)^\dagger=\tau^\dagger\omega^\dagger$;
        \item   $\chi(\omega^\dagger)=\chi(\omega)^{-1}$ for all $\omega\in\shiftsp^*$.
    \end{enumerate}
    A H\"older continuous function $u:\shiftsp\rightarrow(0,+\infty)$ is said to be \emph{symmetric} with respect to $\cdot^\dagger$ if
    \begin{equation} \label{eqn:distconst:dagger}
        V^\dagger_u:=\sup_{\omega\in\shiftsp^*} |S_{\omega}u-S_{\omega^\dagger}u|<+\infty.
    \end{equation}


\section{Restricted Poincar\'e series and its properties} \label{sec:rpoincareseries}
    Let $\leftshift:\shiftsp\rightarrow\shiftsp$ be a topologically transitive SFT, $\fgfuchs$ be a non-cyclic finitely generated free group, and $\chi:\shiftsp^*\rightarrow\fgfuchs$ be a projection with $\chi^{-1}(\identity_\fgfuchs)$ being transitive. Let $u:\shiftsp\rightarrow(0,+\infty)$ be a H\"older continuous function. For each $g\in\fgfuchs$ and $p\in\bbR$, we define what we call the \emph{Poincar\'e series restricted to} $\chi^{-1}(g)$ with exponent $p$ as
    \begin{equation*}
        Z(p|g):=\sum_{\omega\in\chi^{-1}(g)}\exp(-pS_\omega u).
    \end{equation*}
    Since $\chi$ is assumed to be surjective, $\chi^{-1}(g)$ can never be empty, implying that $Z(p|g)>0$ for all $g\in\fgfuchs$ and $p\in\bbR$. For each $g\in\fgfuchs$, define the \emph{Poincar\'e exponent restricted to} $\chi^{-1}(g)$ as
    \begin{equation*}
        \poincareexp(g):=\inf\set{p\in\bbR|Z(p|g)<+\infty}.
    \end{equation*}
    For sufficiently large $p$, we have
    \begin{equation*}
        Z(p|g)\leq \sum_{\omega\in\shiftsp^*}\exp(-pS_\omega u)<+\infty,\,\forall g\in\fgfuchs,
    \end{equation*}
    so $\poincareexp(g)$ is finite for every $g\in\fgfuchs$. In fact, $\poincareexp(g)$ does not depend on $g$, as the following proposition asserts.
    \begin{prop} \label{prop:delta}
        For any $p\in\bbR$ and any $g,g'\in\fgfuchs$, $Z(p|g)<+\infty$ if and only $Z(p|g')<+\infty$. In particular, we have $\poincareexp(g)=\poincareexp(\identity_\fgfuchs)$ for all $g\in\fgfuchs$.
    \end{prop}
    \begin{proof}
        To show the first claim, take $g,g'\in\fgfuchs$ and $p\in\bbR$ arbitrarily. By Proposition~\ref{prop:unifirred}, we can take some $L_{g,g'}>0$ and some map $\theta:\chi^{-1}(g)\rightarrow\chi^{-1}(g')$ such that $\omega$ is a prefix of $\theta(\omega)$ and $|\theta(\omega)|-|\omega|\leq L_{g,g'}$ for all $\omega\in\chi^{-1}(g)$. Then, we immediately have
        \begin{equation} \label{eqn:prop:poincareseries:convergence:proof}
            S_\omega u-V_u\leq S_{\theta(\omega)}u\leq S_{\omega} u+L_{g,g'}\|u\|_\infty,\,\forall\omega\in\chi^{-1}(g),
        \end{equation}
        where we set $\|u\|_\infty:=\sup_{\xi\in\shiftsp}u(\xi)<+\infty$. Also note that for any $\omega'\in\theta(\chi^{-1}(g))\subseteq\chi^{-1}(g')$, every $\omega\in\theta^{-1}(\omega')$ is a prefix of $\omega'$ satisfying $|\omega'|-L_{g,g'}\leq |\omega|\leq |\omega'|$. Hence we have $\#\theta^{-1}(\omega')\leq L_{g,g'}+1$ for all $\omega'\in\chi^{-1}(g')$. Combining this observation with \eqref{eqn:prop:poincareseries:convergence:proof}, we have
        \begin{align*}
            Z(p|g)&=\sum_{\omega'\in\chi^{-1}(g')}\sum_{\omega\in\theta^{-1}(\omega')}\exp(-pS_{\omega}u)\\
            &\leq \exp(|p|\max\set{V_u,L_{g,g'}\|u\|_\infty})\sum_{\omega'\in\chi^{-1}(g')}\sum_{\omega\in\theta^{-1}(\omega')}\exp(-pS_{\theta(\omega)}u)\\
            &\leq \exp(|p|\max\set{V_u,L_{g,g'}\|u\|_\infty})(L_{g,g'}+1)\sum_{\omega'\in\chi^{-1}(g')}\exp(-pS_{\omega'}u)\\
            &=\exp(|p|\max\set{V_u,L_{g,g'}\|u\|_\infty})(L_{g,g'}+1)Z(p|g').
        \end{align*}
        Therefore, $Z(p|g)<+\infty$ if $Z(p|g')<+\infty$. By symmetry, we also have the converse. Thus, the first claim in Proposition \ref{prop:delta} holds. The second claim follows immediately from the first claim.
    \end{proof}
    
    \begin{prop} \label{prop:delta:+}
        $\poincareexp(\identity_\fgfuchs)>0$.
    \end{prop}
    \begin{proof}
        Due to the surjectivity of $\chi$ and the transitivity of $\chi^{-1}(\identity_\fgfuchs)$, every $a\in\alphabet$ admits the existence of $\rho(a)\in\chi^{-1}(\chi(a)^{-2})$ satisfying that $a\rho(a)a\in\shiftsp^*$. Set $\tilde{A}:=\set{a\rho(a)a|a\in\alphabet}$. It is straightforward from its definition that $\tilde{A}$ is a finite subset of $\chi^{-1}(\identity_\fgfuchs)$. Define
        \begin{align*}
            \tilde{A}^*&:=\bigcup_{m=1}^\infty\set{\tau^{(0)}\cdots\tau^{(m-1)}\in\shiftsp^*|\tau^{(j)}\in\tilde{A},\,\forall j\in\set{0,\cdots,m-1}}\subseteq\chi^{-1}(\identity_\fgfuchs),
        \end{align*}
        and also define a map $\iota:\shiftsp^*\setminus\set{\epsilon}\rightarrow\tilde{A}^*$ by setting $\iota(\omega)$ as the unique word $\tau^{(0)}\cdots\tau^{(|\omega|-1)}\in\tilde{A}^*$ satisfying that $[\tau^{(j)}]\subseteq[\omega_j]$ for all $j\in\set{0,\cdots,\left|\omega\right|-1}$. Indeed, we can write
        \begin{equation*}
            \iota(\omega)=\omega_0\rho(\omega_0)\omega_0\omega_1\rho(\omega_1)\omega_1
            \cdots\omega_{|\omega|-1}\rho(\omega_{|\omega|-1})\omega_{|\omega|-1},
        \end{equation*}
        for all $\omega=\omega_0\cdots\omega_{|\omega|-1}\in\shiftsp^*\setminus\set{\epsilon}$. Observe that $\iota(\omega\omega')=\iota(\omega)\iota(\omega')$ for any two $\omega,\omega'\in\shiftsp^*$ satisfying that $\omega\omega'\in\shiftsp^*$. From this observation, we can find a way to recover $\omega$ from a given $\iota(\omega)$ as follows. Note that by our construction, we have $\omega_0=\iota(\omega)_0$. Since $\iota(\omega)=\iota(\omega_0)\iota(\omega_1\cdots\omega_{|\omega|-1})$, we discard $\iota(\omega_0)=\omega_0\rho(\omega_0)\omega_0$ from the beginning of $\iota(\omega)$, then we have $\omega(\omega_1\cdots\omega_{|\omega|-1})$ whose first symbol is precisely $\omega_1$. Repeating this process, we can find out all $\omega_0,\omega_1,\cdots,\omega_{|\omega|-1}$. This shows that $\iota$ is injective. Therefore, for all $p\in\bbR$, we have
        \begin{equation*}
            Z(p|\identity_\fgfuchs) \geq \sum_{\omega\in\tilde{A}^*}\exp(-pS_\omega u)\geq\sum_{\omega\in\shiftsp^*}\exp(-pS_{\iota(\omega)} u).
        \end{equation*}
        By the definition of $\iota(\omega)$, we clearly have $|\omega|\leq |\iota(\omega)|\leq |\omega|\cdot\max_{\tau\in\tilde{A}}|\tau|$. Hence, in order to show Proposition \ref{prop:delta:+}, one only needs to show the existence of $p_0>0$ such that
        \begin{equation*}
            \sum_{\omega\in\shiftsp^*}\exp(-p_0|\omega|)=+\infty.
        \end{equation*}
        
        Since the SFT $\leftshift_\fgfuchs:\partial\fgfuchs\rightarrow\partial\fgfuchs$ given by the alphabet $\fgfuchs_0$ and the incidence matrix $\incmat_\fgfuchs$ we described in Subsection~\ref{subsec:freegroups} is clearly topologically mixing, there is some $p_\fgfuchs>0$ such that
        \begin{equation*}
            \sum_{g\in\fgfuchs} \exp(-p_\fgfuchs|g|)=+\infty.
        \end{equation*}
        Let $l_\chi:=\max_{a\in\alphabet}\max_{e\in\fgfuchs_0}\min_{\omega\in[a]\cap\chi^{-1}(e)}|\omega|$. Then, for any $g\in\fgfuchs\setminus\set{\identity_\fgfuchs}$, there is some word $\omega\in\chi^{-1}(g)$ such that $|\omega|\leq l_\chi |g|$. Thus, we have
        \begin{equation*}
            \sum_{\omega\in\shiftsp^*} \exp(-\frac{p_\fgfuchs}{l_\chi}|\omega|)\geq \sum_{g\in\fgfuchs} \exp(-p_\fgfuchs|g|)=+\infty.
        \end{equation*}
        Since $p_\fgfuchs/l_\chi$ is clearly positive, the proof is complete.
    \end{proof}

    We next show that the restricted Poincar\'e series is almost supermultiplicative in the following sense.
    \begin{prop} \label{prop:poincareseries:supadd}
        For any $p\in\bbR$, there exists some $C_+(p)\in(1,+\infty)$ such that
        \begin{equation} \label{eqn:prop:poincareseries:supadd}
            Z(p|g)Z(p|g')\leq C_+(p)Z(p|gg'),\,\forall g,g'\in\fgfuchs.
        \end{equation}
    \end{prop}
    
    \begin{proof}
        For $p\leq 0$, we have seen in Proposition \ref{prop:delta} that $Z(p|g)=+\infty$ for all $g\in\fgfuchs$. In this case, our claim holds trivially. Henceforth, we assume that $p>0$. We also fix a finite transitive subset $\mathcal{I}$ of $\chi^{-1}(\identity_\fgfuchs)$.
        
        Given any $g,g'\in\fgfuchs$, we define a map $\theta:\chi^{-1}(g)\times\chi^{-1}(g')\rightarrow\chi^{-1}(gg')\times\chi^{-1}(\identity_\fgfuchs)$ as follows. For any $(\omega,\omega')\in\chi^{-1}(g)\times\chi^{-1}(g')$, take the shortest prefix $\omega^{(0)}$ of $\omega$ such that $\chi(\omega^{(0)})=g$. Let $\omega=\omega^{(0)}\omega^{(1)}$. Then, $\chi(\omega^{(1)})=\identity_\fgfuchs$. Then, we pick some $\rho\in\mathcal{I}$ such that $\omega^{(0)}\rho\omega'\in\shiftsp^*$. We set
        \begin{equation*}
            \theta(\omega,\omega'):=(\omega^{(0)}\rho\omega',\omega^{(1)})\in \chi^{-1}(gg')\times\chi^{-1}(\identity_\fgfuchs).
        \end{equation*}
        For later convenience, the two components of $\theta(\omega,\omega')$ will be denoted by $\theta_1(\omega,\omega')=\omega^{(0)}\rho\omega'$ and $\theta_2(\omega,\omega')=\omega^{(1)}$. Let $L_{\mathcal{I}}:=\max_{\rho'\in\mathcal{I}}|\rho'|$, and observe that
        \begin{align*}
            \sum_{j=1}^2S_{\theta_j(\omega,\omega')}u\leq S_{\omega}u+S_{\omega'}u+L_{\mathcal{I}}\|u\|_\infty+V_u.
        \end{align*}

        Note that given $\theta(\omega,\omega')=(\theta_1(\omega,\omega'),\theta_2(\omega,\omega'))$, we can recover $\omega$ as follows. Let $\omega^{(0)}$ be the shortest prefix of $\theta_1(\omega,\omega')$ such that $\chi(\omega^{(0)})=g$. It is easy to see from our construction that $\omega$ is bound to be $\omega^{(0)}\theta_2(\omega,\omega')$. Since $\omega^{(0)}$ is unique and the joining block $\rho$ in the middle has its length at most $L_{\mathcal{I}}$, the number of suffixes of $\theta_1(\omega,\omega')$ which can possibly be $\omega'$ cannot exceed $L_{\mathcal{I}}+1$. Therefore, we have
        \begin{equation*}
            \#\theta^{-1}(\omega'',\omega''')\leq L_{\mathcal{I}}+1,
        \end{equation*}
        for all $(\omega'',\omega''')\in \chi^{-1}(gg')\times\chi^{-1}(\identity_\fgfuchs)$.

        Using the facts we have so far obtained, we derive that for $p>0$,
        \begin{align*}
            Z(p|g)Z(p|g')&=\sum_{(\omega,\omega')\in\chi^{-1}(g)\times\chi^{-1}(g')} \exp(-p(S_\omega u+S_{\omega'}u))\\
            &\leq \frac{\tilde{C}_+(p)}{L_{\mathcal{I}}+1}\sum_{(\omega,\omega')\in\chi^{-1}(g)\times\chi^{-1}(g')}\exp(-p\sum_{j=1}^2S_{\theta_j(\omega,\omega')}u)\\
            &\leq \tilde{C}_+(p)Z(p|gg')Z(p|\identity_\fgfuchs),
        \end{align*}
        where
        \begin{equation*}
            \tilde{C}_+(p):=(L_{\mathcal{I}}+1)\exp(pL_{\mathcal{I}}\sup(u(\shiftsp_0))+pV_u).
        \end{equation*}
        When $Z(p|g)=+\infty$ for all $g\in\fgfuchs$, our claim in Proposition \ref{prop:poincareseries:supadd} clearly holds. Otherwise, by Proposition \ref{prop:delta}, we have $Z(p|g)<+\infty$ for all $g\in\fgfuchs$, which in particular gives $Z(p|\identity_\fgfuchs)<+\infty$. Therefore, setting $C_+(p):=\tilde{C}_+(p)Z(p|\identity_\fgfuchs)$, we have our desired claim.
    \end{proof}

    It is clear from the definition that $Z(p|g)$ decreases as $p>\poincareexp(\identity_\fgfuchs)$ increases for every $g\in\fgfuchs$. The next proposition shows that when $p$ increases, $Z(p|g)$ decreases at a higher speed, or increases at a lower speed, as the length $|g|$ of $g$ tends to infinity.
    \begin{prop} \label{prop:decrateZ}
        Let $p'>p>\poincareexp(\identity_\fgfuchs)$. Then, we have
        \begin{equation} \label{eqn:prop:decrateZ}
            \sum_{m=1}^\infty\frac{Z(p'|x_0\cdots x_{m-1})}{Z(p|x_0\cdots x_{m-1})}<+\infty,\, \forall x\in\partial\fgfuchs.
        \end{equation}
    \end{prop}
    \begin{proof}
        Fix $p'>p>\poincareexp(\identity_\fgfuchs)$. Observe that for any $\omega\in\shiftsp^*$, we have
        \begin{equation*}
            |\chi(\omega)|\leq\sum_{j=0}^{|\omega|-1}|\chi(\omega_j)|\leq |\omega|\lambda_1(\chi),
        \end{equation*}
        where $\lambda_1(\chi) :=\max_{a\in\alphabet}|\chi(a)|<+\infty$. Since $\chi$ is surjective, we also have $\lambda_1(\chi)\geq 1$. Then, for any $g\in\fgfuchs$, we have
        \begin{align*}
            Z(p'|g)&\leq \sum_{\omega\in\chi^{-1}(g)}\exp((p-p')S_\omega u+(p'-p)V_u)\exp(-pS_\omega u)\\
            &\leq \sum_{\omega\in\chi^{-1}(g)}\exp((p-p')|\omega|\min_{\xi\in\shiftsp} u(\xi)+(p'-p)V_u)\exp(-pS_\omega u)\\
            &\leq \exp((p'-p)V_u+(p-p')|g|\min_{\xi\in\shiftsp} u(\xi)/\lambda_1(\chi))Z(p|g).
        \end{align*}
        Hence, for all $x\in\partial\fgfuchs$, we have
        \begin{align*}
            \sum_{m=1}^\infty\frac{Z(p'|x_0\cdots x_{m-1})}{Z(p|x_0\cdots x_{m-1})}\leq \exp((p'-p)V_u)\sum_{m=1}^\infty \exp((p-p')m\min_{\xi\in\shiftsp} u(\xi)/\lambda_1(\chi)).
        \end{align*}
        Since $(p-p')\min_{\xi\in\shiftsp} u(\xi)/\lambda_1(\chi)<0$, we can conclude that \eqref{eqn:prop:decrateZ} holds.
    \end{proof}

    As a corollary, we have the following fact.
    \begin{cor} \label{cor:Zgeodseries:conv:symmetric}
        Suppose that, for some involution $\cdot^\dagger:\shiftsp^*\rightarrow\shiftsp^*$, the $\fgfuchs$-extension of $\leftshift:\shiftsp\rightarrow\shiftsp$ given by $\chi$ and $u:\shiftsp\rightarrow(0,+\infty)$ are both symmetric with respect to $\cdot^\dagger$. Then, $\sum_{m=1}^\infty Z(p|x_0\cdots x_{m-1})<+\infty$ for all $x\in\partial\fgfuchs$ and all $p>\poincareexp(\identity_\fgfuchs)$.
    \end{cor}
%
%
%
    
    \begin{proof}
        Fix an arbitrary $p>\poincareexp(\identity_\fgfuchs)$ and take  $p'\in(\poincareexp(\identity_\fgfuchs),p)$. Then, by the symmetry of $\fgfuchs$-extension with respect to $\cdot^\dagger$, we have for any $h\in\fgfuchs$,
        \begin{align*}
            Z(p'|h^{-1})&=\sum_{\omega\in\chi^{-1}(h^{-1})}\exp(-p'S_\omega u)= \sum_{\omega\in\chi^{-1}(h)}\exp(-p'S_{\omega^\dagger}u)\\
            &\leq\exp(p'V^\dagger_u)\sum_{\omega\in\chi^{-1}(h)}\exp(-p'S_{\omega}u)=\exp(p'V^\dagger_u) Z(p'|h),
        \end{align*}
        where $V^\dagger_u$ is defined as in \eqref{eqn:distconst:dagger}. Thus the symmetry of $u$ with respect to $\cdot^\dagger$ gives
        \begin{equation*} 
            \sup_{h\in\fgfuchs}\frac{Z(p'|h^{-1})}{Z(p'|h)}\leq \exp(p'V^\dagger_u)<+\infty.
        \end{equation*}
        
        For any $g\in\fgfuchs$, the almost supermultiplicativity of the restricted Poincar\'e series stated in Proposition \ref{prop:poincareseries:supadd} yields
        \begin{equation*}
            Z(p'|g)Z(p'|g^{-1})\leq C_+(p')Z(p'|\identity_\fgfuchs),
        \end{equation*}
        which further gives
        \begin{equation*}
            Z(p'|g)^2\leq Z(p'|g)Z(p'|g^{-1}) \sup_{h\in\fgfuchs}\frac{Z(p'|h^{-1})}{Z(p'|h)}\leq C_+(p')Z(p'|\identity_\fgfuchs)\exp(p'V^\dagger_u).
        \end{equation*}
        Consequently, we have $\sup_{g\in\fgfuchs}Z(p'|g)<+\infty$. Combining this with Proposition \ref{prop:decrateZ} and the fact that $p>p'>\poincareexp(\identity_\fgfuchs)$, we have
        \begin{align*}
            \sum_{m=1}^\infty Z(p|x_0\cdots x_{m-1})&=\sum_{m=1}^\infty\frac{Z(p|x_0\cdots x_{m-1})}{Z(p'|x_0\cdots x_{m-1})} Z(p'|x_0\cdots x_{m-1})\\
            &\leq\sum_{m=1}^\infty\frac{Z(p|x_0\cdots x_{m-1})}{Z(p'|x_0\cdots x_{m-1})} \sup_{g\in\fgfuchs}Z(p'|g)<+\infty,
        \end{align*}
        for all $x\in\partial\fgfuchs$. The proof is thus complete.
    \end{proof}



%

\section{Proofs of Theorem \ref{mainthm:dimspec:general} and Theorem \ref{mainthm:dimspec:symmetric}} \label{sec:mainproof}
    We prove Theorem \ref{mainthm:dimspec:general} and Theorem \ref{mainthm:dimspec:symmetric} in this section. We fix a topologically transitive SFT $\leftshift:\shiftsp\rightarrow\shiftsp$ given by some alphabet $\alphabet$ and some incidence matrix $\incmat$, a finitely generated free group $\fgfuchs$ and a projection $\chi:\shiftsp^*\rightarrow\fgfuchs$ for which $\chi^{-1}(\identity_\fgfuchs)$ is transitive.

\subsection{Proof of Theorem \ref{mainthm:dimspec:general}}
    First we prove Theorem \ref{mainthm:dimspec:general}. The argument we shall use is inspired by the proofs of \cite[Theorem 1.5]{gjk22} and \cite[Theorem 1.2]{liu}.

    The following lemma plays a role similar to \cite[Lemma 3.1]{liu}. We postpone its proof to the appendix.
    \begin{lma} \label{lma:disjcylinders}
        There is a finite transitive set $\mathcal{I}\subseteq\chi^{-1}(\identity_\fgfuchs)$ satisfying that $[\rho]\cap[\rho']=\varnothing$ for any two distinct $\rho,\rho'\in\mathcal{I}$.
    \end{lma}

    \begin{proof} [Proof of Theorem \ref{mainthm:dimspec:general}]
        The fact that $\poincareexp(\identity_\fgfuchs)>0$ has been given in Proposition \ref{prop:delta:+}. Hence it only remains to show $\dim_u(\unifescset_x)\geq \poincareexp(\identity_\fgfuchs)$ for all $x\in\partial\fgfuchs$. Fix an arbitrary positive $p<\poincareexp(\identity_\fgfuchs)$ and an arbitrary $x\in\partial\fgfuchs$. We aim to prove that $\dim_u(\unifescset_x)\geq p$, and our proof will then be complete.

        Let the free group $\fgfuchs$ be freely generated by $\set{e_1,\cdots,e_n}\subseteq\fgfuchs$ and set $\fgfuchs_0:=\set{e_1,e_1^{-1},\cdots,e_n,e_n^{-1}}$. By Proposition \ref{prop:unifirred}, there are $2n$ non-empty words
        \begin{equation*}
            \tau^{(1)}\in\chi^{-1}(e_1),\,\tau^{(-1)}\in\chi^{-1}(e_1^{-1}),\,\cdots,\,\tau^{(n)}\in\chi^{-1}(e_n),\,\tau^{(-n)}\in\chi^{-1}(e_n^{-1})
        \end{equation*}
        of which the cylinder sets are pairwise disjoint. Take a finite transitive subset $\mathcal{I}$ of $\chi^{-1}(\identity_\fgfuchs)$ and a finite subset $\mathcal{W}$ of $\chi^{-1}(\identity_\fgfuchs)$ such that 
        \begin{equation} \label{eqn:prop:lobound:proof1}
            \log\sum_{\tau\in\mathcal{W}}\exp(-pS_\tau u)> p\|u\|_\infty\left(\max_{s\in\set{\pm 1}}\max_{j\in\set{1,\cdots,n}}|\tau^{(sj)}|+2\max_{\rho\in\mathcal{I}}|\rho|\right)+3pV_u,
        \end{equation}
        where $\|u\|_\infty:=\max_{\xi\in\shiftsp}u(\xi)$. Now for each $j\in\set{1,\cdots,n}$, each $s\in\set{\pm1}$ and each $a\in\alphabet$, we define
        \begin{equation*}
            \mathcal{A}(e_j^s,a):=\Set{\tau^{(sj)}\rho\tau\rho'\in\shiftsp^*|\rho,\rho'\in\mathcal{I},\,\tau\in\mathcal{W},\,\rho'a\in\shiftsp^*}\subseteq\chi^{-1}(e_j^s).
        \end{equation*}
        By \eqref{eqn:prop:lobound:proof1} and the transitivity of $\mathcal{I}$, we have
        \begin{equation} \label{eqn:prop:lobound:proof2}
            \sum_{\omega\in\mathcal{A}(e_j^s,a)}\exp(-pS_\omega u)\geq \exp(pV_u)
        \end{equation}
        holds for all $j\in\set{1,\cdots,n}$, $s\in\set{\pm1}$ and $a\in\alphabet$. Further set
        \begin{equation*}
            \mathcal{A}(e_j^s):=\bigcup_{a\in\alphabet}\mathcal{A}(g_j^s,a)
        \end{equation*}
        for each $j\in\set{1,\cdots,n}$ and $s\in\set{\pm1}$.

        We now define a Borel probability measure $\mu^p$ on $\shiftsp$ by induction as follows. First note that, for any $j\in\set{1,\cdots,n}$ and $s\in\set{\pm1}$, our construction of $\mathcal{A}(e^s_j)$ guarantees that all the elements of $\mathcal{A}(e^s_j)$ begin with the same symbol $\tau^{(sj)}_1$. From this observation, for every integer $k\geq 0$, we can let $a_k$ to be the first symbol of any of the elements of $\mathcal{A}(x_k)$. For each $\omega^{(0)}\in\mathcal{A}(x_0,a_1)$, we define
        \begin{equation*}
            \mu^p([\omega^{(0)}]) :=\frac{\exp(-pS_{\omega^{(0)}}u)}{\sum_{\omega\in\mathcal{A}(x_0,a_1)}\exp(-pS_\omega u)}.
        \end{equation*}
        For any integer $k\geq 1$ and any $\omega^{(0)}\cdots\omega^{(k)}\in\shiftsp^*$ satisfying that $\omega^{(j)}\in\mathcal{A}(x_j,a_{j+1})$ for all $j\in\set{0,\cdots,k}$, we define
        \begin{equation} \label{eqn:prop:lobound:proof3}
            \mu^p([\omega^{(0)}\cdots\omega^{(k)}]) :=\frac{\exp(-pS_{\omega^{(0)}\cdots\omega^{(k-1)}\omega^{(k)}}u)\mu^p([\omega^{(0)}\cdots\omega^{(k-1)}])}{\sum_{\omega\in\mathcal{A}(x_k,a_{k+1})}\exp(-pS_{\omega^{(0)}\cdots\omega^{(k-1)}\omega}u)}.
        \end{equation}
        Lemma \ref{lma:disjcylinders} implies that $[\omega^{(0)}\cdots\omega^{(k-1)}\omega^{(k)}]\cap[\omega^{(0)}\cdots\omega^{(k-1)}\tau^{(k)}]=\varnothing$ if $\omega^{(k)},\tau^{(k)}\in\mathcal{A}(x_k,a_{k+1})$ are distinct. Thus, by Kolmogorov's consistency theorem, $\mu^p$ is well-defined.
        
        It follows from our construction of $\mu^p$ that the support of $\mu^p$ is contained in the compact set
        \begin{equation*}
            K:=\bigcap_{k=1}^\infty \bigcup_{\omega^{(0)}\in\mathcal{A}(x_0,a_1),\cdots,\omega^{(k-1)}\in\mathcal{A}(x_{k-1},a_k)}[\omega^{(0)}\cdots\omega^{(k-1)}].
        \end{equation*}
        Clearly, every $\xi\in K$ can be written as $\omega^{(0)}\omega^{(1)}\cdots$, where $\omega^{(k)}\in\mathcal{A}(x_k,a_{k+1})$ for every integer $k\geq 0$. Observe that
        \begin{equation*}
            \max_{k\geq 0}|\omega^{(k)}|<+\infty
        \end{equation*}
        because $\bigcup_{j\in\set{1,\cdots,n}}\bigcup_{s\in\set{\pm 1}}\mathcal{A}(e^s_j)$ is a finite set.
        Hence, for each integer $m\geq 1$, let $k(m)$ be the least non-negative integer for which $\xi_0\cdots\xi_{m-1}$ is a prefix of $\omega^{(0)}\cdots\omega^{k(m)}$, and it is not hard to see that
        \begin{equation*}
            \sup_{m\geq 1}\big||\chi(\xi_0\cdots\xi_{m-1})|-k(m)\big|\leq \sup_{m\geq 1}\big||\chi(\xi_0\cdots\xi_{m-1})|-|\chi(\omega^{(0)}\cdots\omega^{k(m)})|\big|<+\infty,
        \end{equation*}
        and
        \begin{align*}
            &\quad\sup_{m\geq 1}\big||\chi(\xi_0\cdots\xi_{m-1})\wedge x|-k(m)\big|\\
            &\leq \sup_{m\geq 1}\big||\chi(\xi_0\cdots\xi_{m-1})\wedge x|-|\chi(\omega^{(0)}\cdots\omega^{k(m)})\wedge x|\big|<+\infty.
        \end{align*}
        From these inequalities, we have $\xi\in\unifescset_x$. Since $\xi$ is taken arbitrarily from $K$, we see that $K\subseteq \unifescset_x$, which implies that $\mu^p(\unifescset_x)=1$.

        For each integer $k\geq 1$ and $\omega^{(0)}\cdots\omega^{(k-1)}\in\shiftsp^*$ satisfying $\omega^{(j)}\in\mathcal{A}(x_j,a_{j+1})$ for all $j\in\set{0,\cdots,k-1}$, we deduce from \eqref{eqn:prop:lobound:proof2} that
        \begin{align*}
            &\quad\sum_{\omega\in\mathcal{A}(x_k,a_{k+1})}\exp(-pS_{\omega^{(0)}\cdots \omega^{(k-1)}\omega}u)\\
            &\geq \exp(-pS_{\omega^{(0)}\cdots \omega^{(k-1)}}u-pV_u)\sum_{\omega\in\mathcal{A}(x_k,a_{k+1})}\exp(-pS_{\omega}u)\\
            &\geq \exp(-pS_{\omega^{(0)}\cdots \omega^{(k-1)}}u).
        \end{align*}
        Hence, combining this with how we constructed $\mu^p$, we have for all integers $k\geq 1$,
        \begin{align*}
            &\quad \max_{\omega^{(0)}\in\mathcal{A}(x_0,a_1),\cdots,\omega^{(k)}\in\mathcal{A}(x_k,a_{k+1})} \frac{\mu^p([\omega^{(0)}\cdots\omega^{(k)}])}{\exp(-pS_{\omega^{(0)}\cdots\omega^{(k)}}u)}\\
            &=\max_{\omega^{(0)}\in\mathcal{A}(x_0,a_1),\cdots,\omega^{(k-1)}\in\mathcal{A}(x_{k-1},a_k)}\frac{\mu^p([\omega^{(0)}\cdots\omega^{(k-1)}])}{\sum_{\omega\in\mathcal{A}(x_k,a_{k+1})}\exp(-pS_{\omega^{(0)}\cdots\omega^{(k-1)}\omega}u)}\\
            &\leq \max_{\omega^{(0)}\in\mathcal{A}(x_0,a_1),\cdots,\omega^{(k-1)}\in\mathcal{A}(x_{k-1},a_k)}\frac{\mu^p([\omega^{(0)}\cdots\omega^{(k-1)}])}{\exp(-pS_{\omega^{(0)}\cdots \omega^{(k-1)}}u)},
        \end{align*}
        where the first equality follows from \eqref{eqn:prop:lobound:proof3}. As a consequence, we have
        \begin{equation*}
            \max_{\omega^{(0)}\in\mathcal{A}(x_0,a_1),\cdots,\omega^{(k)}\in\mathcal{A}(x_k,a_{k+1})} \frac{\mu^p([\omega^{(0)}\cdots\omega^{(k)}])}{\exp(-pS_{\omega^{(0)}\cdots\omega^{(k)}}u)}\leq\max_{\omega^{(0)}\in\mathcal{A}(x_0,a_1)} \frac{\mu^p([\omega^{(0)}])}{\exp(-pS_{\omega^{(0)}}u)},
        \end{equation*}
        for every integer $k\geq 0$. Since we have seen at the end of the previous paragraph that $\mu^p([\omega])=1$, by the mass distribution principle, we thus have $\dim_u(\unifescset_x)\geq p$. Since $p$ is an arbitrary positive number less than $\poincareexp(\identity_\fgfuchs)$, the proof is complete.
    \end{proof}
    
\subsection{Proof of Theorem \ref{mainthm:dimspec:symmetric}} \label{subsec:dimspec:symmetric}
    We utilize the results we obtained in Section \ref{sec:rpoincareseries} to show Theorem \ref{mainthm:dimspec:symmetric} as follows.
    \begin{proof} [Proof of Theorem \ref{mainthm:dimspec:symmetric}]
        Fix an arbitrary $x\in\partial\fgfuchs$ and $p>\poincareexp(\identity_\fgfuchs)> 0$. For each $r\geq 0$, we define
        \begin{equation*}
            \mathcal{S}(x,r):=\Set{\omega\in\shiftsp^*| \left|\chi(\omega)\right|-|\chi(\omega)\wedge x|\leq r}.
        \end{equation*}
        Then, for any $r\geq 0$, we employ the almost multiplicativity of the restricted Poincar\'e series in Proposition \ref{prop:poincareseries:supadd} to deduce
        \begin{align*}
            \sum_{\omega\in\mathcal{S}(x,r)}\exp(-pS_\omega u)&\leq\sum_{m=0}^\infty \sum_{g\in\fgfuchs:|g|\leq r}\sum_{\omega\in\chi^{-1}(x_0\cdots x_{m-1}g)}\exp(-pS_\omega u)\\
            &=\sum_{m=0}^\infty \sum_{g\in\fgfuchs:|g|\leq r}Z(p|x_0\cdots x_{m-1}g)\\
            &\leq \sum_{m=0}^\infty \sum_{g\in\fgfuchs:|g|\leq r}C_+(p)\frac{Z(p|x_0\cdots x_{m-1})}{Z(p|g^{-1})}\\
            &=C_+(p)\sum_{g\in\fgfuchs:|g|\leq r}\frac{1}{Z(p|g^{-1})} \sum_{m=0}^\infty Z(p|x_0\cdots x_{m-1}),
        \end{align*}
        where we set $x_0\cdots x_{-1}$ to be $\identity_\fgfuchs$ and take $C_+(p)$ as in Proposition \ref{prop:poincareseries:supadd}. With the symmetries we assumed, we get from Corollary \ref{cor:Zgeodseries:conv:symmetric} that $\sum_{m=1}^\infty Z(p|x_0\cdots x_{m-1})<+\infty$. Therefore, for all $r\geq 0$, we have
        \begin{equation*}
            \sum_{\omega\in\mathcal{S}(x,r)}\exp(-pS_\omega u)<+\infty.
        \end{equation*}

        Take an arbitrary $\xi\in\accumset_x$. For every integer $k\geq 1$, we take $m_k$ as the least positive integer such that $x_0\cdots x_{k-1}$ is a prefix of $\chi(\xi_0\cdots \xi_{m_k-1})$. Then, for each positive integer $k$, since $x$ is not a prefix of $\chi(\xi_0\cdots\xi_{m_k-2})$, we have
        \begin{align*}
            |\chi(\xi_0\cdots\xi_{m_k-1})|-|\chi(\xi_0\cdots\xi_{m_k-1})\wedge x|&\leq |\chi(\xi_0\cdots\xi_{m_k-1})|-|\chi(\xi_0\cdots\xi_{m_k-2})|\\
            &\leq |\chi(\xi_{m_k-1})|\leq\lambda_1(\chi):=\max_{a\in\alphabet}|\chi(a)|.
        \end{align*}
        Equivalently, we can say $\xi_0\cdots\xi_{m_k-1}\in\mathcal{S}(x,\lambda_1(\chi))$ for all integer $k\geq 1$. Therefore, for all $L>0$, $\set{[\omega]|\omega\in\mathcal{S}(x,\lambda_1(\chi)),\,|\omega|\geq L}$ is a covering of $\accumset_x$. Combining this with the convergence of $\sum_{\omega\in\mathcal{S}(x,\lambda_1(\chi))}\exp(-pS_\omega u)$ we showed in the previous paragraph, we can see that $\dim_u(\accumset_x)\leq p$. Since $p$ is an arbitrary real number greater than $\poincareexp(\identity_\fgfuchs)$, we have $\dim_u(\accumset_x)\leq \poincareexp(\identity_\fgfuchs)$. Combining this with the lower bound we obtained in Theorem \ref{mainthm:dimspec:general}, we conclude that \eqref{eqn:dimspec:eq} holds.
    \end{proof}

\section{Application to geodesic flows on free group covers of Schottky surfaces} \label{sec:appl}
In this section, we firstly give the definition as well as some basic facts about Schottky groups, and then prove Theorem \ref{mainthm:geodflow} stated in the introduction.
\subsection{Preliminaries on Schottky groups} \label{subsec:schottky:def}
    Our definition of Schottky groups is aligned with the definition in \cite[Section 15.1]{borthwick}. Let $\hyperbsp$ be the open unit disk in $\bbC$, equipped with the hyperbolic metric $d_h$. Let $\alphabet$ be a finite set whose cardinality is an even number no less than $6$, and $\bar{\cdot}:\alphabet\rightarrow\alphabet$ be an involution with no fixed points. Associate to each $a\in\alphabet$ a closed subset $\mathcal{H}_a$ of $\hyperbsp$ whose boundary in $\hyperbsp$, denoted by $\mathcal{S}_a$, is a complete geodesic in $\hyperbsp$. Further associate to each $a\in\alphabet$ an orientation-preserving isometry mapping $\hyperbsp$ onto $\hyperbsp$ and mapping $\mathcal{H}_{\bar{a}}$ onto $(\hyperbsp\setminus\mathcal{H}_a)\cup\mathcal{S}_a$. Slightly abusing the notation, we shall denote this isometry by $a$. In addition, the isometries are taken so that $a$ and $\bar{a}$ are inverse to each other for all $a\in\alphabet$. Then $\alphabet$, which is now considered as a subset of the group $\Isom_+(\hyperbsp)$ of orientation-preserving isometries on $\hyperbsp$, generates a Fuchsian group $G$. This group $G$ is called a Schottky group.

    The following assertions in this and the next two paragraphs are all well-known facts about Schottky groups; see e.g. Section 15.1 and Section 15.2 of \cite{borthwick}. The closed polygon
    \begin{equation*}
        \fundpoly :=\left(\hyperbsp\setminus\bigcup_{a\in\alphabet}\mathcal{H}_a\right)\cup\bigcup_{a\in\alphabet}\mathcal{S}_a
    \end{equation*}
    is a fundamental domain of $G$. The boundary of $\fundpoly$ in $\hyperbsp$ is precisely the union $\bigcup_{a\in\alphabet}\mathcal{S}_a$ of complete geodesics. These geodesics are paired by the elements of $\alphabet$; for every $a\in\alphabet$, the image $a\mathcal{S}_{\bar{a}}$ of the geodesic $\mathcal{S}_{\bar{a}}$ under the isometry $a$ is exactly the geodesic $\mathcal{S}_a$. Hence, the elements of $\alphabet$ are often called side-pairing transformations.

    The Schottky group $G$ is a free group, and $\alphabet$ is a balanced free generating set of $G$. Hence, every element of $G$ can be uniquely written as a reduced subword over $\alphabet$; we represent the identity $\identity_G$ by the empty word. As in Subsection \ref{subsec:freegroups}, we will not distinguish the group elements of $G$ from their corresponding reduced words over $\alphabet$.

    Aside from the one-to-one correspondence between group elements of $G$ and reduced words over $\alphabet$, there is also a natural one-to-one correspondence between the limit points of $G$ and the reduced sequences over $\alphabet$. A limit point of a Fuchsian group $H$ is an accumulation point of the orbit $H0=\set{h0|h\in H}$, and the limit set $\limitset(H)$ of $H$ is the set of all the limit points of $H$. As the orbit $H0$ cannot accumulate in $\hyperbsp$, $\limitset(G)$ is a subset of the unit circle $\partial\hyperbsp$. For a Schottky group $G$, we can associate a limit point $\pi(\xi)$ of $G$ to each reduced sequence $\xi$ over $\alphabet$. For each reduced $\omega=\omega_0\cdots\omega_{|\omega|-1}\in G\setminus\set{\identity_G}$, we define
    \begin{equation*}
        \mathcal{H}_\omega:=\omega\mathcal{H}_{\overline{\omega_{|\omega|-1}}}.
    \end{equation*}
    Then for each reduced sequence $\xi$ over $\alphabet$, it is known that $(\overline{\mathcal{H}_{\xi_0\cdots\xi_{m-1}}})_{m\geq 1}$ is descending, where $\overline{\mathcal{H}_{\xi_0\cdots\xi_{m-1}}}$ is the closure of $\mathcal{H}_{\xi_0\cdots\xi_{m-1}}$ taken in the closed disk $\hyperbsp\cup\partial\hyperbsp$. Moreover, $\bigcap_{m=1}^\infty\overline{\mathcal{H}_{\xi_0\cdots\xi_{m-1}}}$ is a singleton, whose only element is set to be $\pi(\xi)$. Let $\shiftsp_G$ be the set of all reduced sequences over $\alphabet$. Then, the coding map $\pi:\shiftsp_G\rightarrow\limitset(G)$ is a homeomorphism, where the topology of $\shiftsp_G$ is generated by all the cylinder sets. Indeed, we have $\pi([\omega])=\overline{\mathcal{H}_{\omega}}\cap \limitset(G)$ for all $\omega\in G\setminus\set{\identity_G}$. The natural left shift $\leftshift_G:\shiftsp_G\rightarrow\shiftsp_G$ on $\shiftsp_G$ can also be passed to $\limitset(G)$. Following the idea of Bowen and Series in \cite{BowSer79}, one may define the Bowen-Series map $f:\bigcup_{a\in\alphabet}I_a\rightarrow\partial\hyperbsp$ by setting $f\big|_{I_a}:=a^{-1}$ for all $a\in\alphabet$, where we define $I_a:=\overline{\mathcal{H}_a}\cap\partial\hyperbsp$ for all $a\in\alphabet$. Then, the range of $f\big|_{\limitset(G)}$ is $\limitset(G)$, and $f\big|_{\limitset(G)}:\limitset(G)\rightarrow\limitset(G)$ is conjugate to $\leftshift_G:\shiftsp_G\rightarrow\shiftsp_G$ via $\pi$, i.e. $f\big|_{\limitset(G)}=\pi\circ\leftshift_G\circ\pi^{-1}$.

    Define $u_0:\shiftsp_G\rightarrow\bbR$ by letting
    \begin{equation*}
        u_0(\xi):=\log|f'(\pi(\xi))|,\,\forall\xi\in\shiftsp_G.
    \end{equation*}
    The H\"older continuity of $u_0$ can be seen from \cite[Lemma 3.3]{KesStr04}. The Bowen-Series map is known to be eventually expanding, i.e. $\inf_{\xi\in\shiftsp_G}|(f^m)'(\pi(\xi))|>1$ for sufficiently large $m$, hence implying that the $m$-th Birkhoff average
    \begin{equation*}
        u:=\frac{1}{m}S_mu_0
    \end{equation*}
    is positive for sufficiently large $m$. In addition, $u$ is a linear combination of H\"older continuous functions, so $u$ is also H\"older continuous. Moreover, $u$ and $u_0$ are cohomologous to each other, by which we mean that there is some continuous function $v:\shiftsp\rightarrow\bbR$ such that
    \begin{equation} \label{eqn:cohom}
        u=u_0+v\circ \leftshift_G-v.
    \end{equation}
    When $m=1$, we have $u=u_0$, so we can take $v=0$. For $m\geq 2$, we take
    \begin{equation*}
        v=\frac{1}{m}\sum_{k=0}^{m-2}(m-1-k)u_0\circ\leftshift_G^k.
    \end{equation*}
    Then direct calculations give \eqref{eqn:cohom}. An immediate consequence of \eqref{eqn:cohom} is
    \begin{equation*}
        \sup_{m\in\bbN} \sup_{\xi\in\limitset(G)} |S_m u(\xi)-S_mu_0(\xi)|\leq 2\|v\|_\infty<+\infty.
    \end{equation*}
    Combining this with the definition of $u_0$, it is not hard to derive that
    \begin{equation} \label{eqn:hausdim:udim}
        \hausdim(\pi(E))=\dim_u(E),
    \end{equation}
    for all $E\subseteq\shiftsp_G$, where the Hausdorff dimension on the left-hand side of the equation is with respect to the spherical metric on $\partial\hyperbsp$. The Birkhoff sums of $u$ are also related to the hyperbolic metric $d_h$ on $\hyperbsp$. Since our definition of Schottky groups does not allow the existence of parabolic elements, it is known \cite[Subsection 3.1.2]{KesStr04} that
    \begin{equation} \label{eqn:birksum:hyperbmetric}
        \Delta:=\sup_{\xi\in\shiftsp} \sup_{m\in\bbN} |d_h(0,\xi_0\cdots\xi_{m-1}0)-S_mu(\xi)|<+\infty.
    \end{equation}
    
\subsection{Proof of Theorem \ref{mainthm:geodflow}} \label{subsec:appl:proof}
    We show how to prove Theorem \ref{mainthm:geodflow} using Theorem \ref{mainthm:dimspec:symmetric}.
    \begin{proof} [Proof of Theorem \ref{mainthm:geodflow}]
        As we remarked in the introduction, we only need to prove the first three equalities, i.e.
        \begin{equation*}
            \hausdim(\accumset_x(\nsubg))= \hausdim(\escset_x(\nsubg))=\hausdim(\unifescset_x(\nsubg))=\poincareexp(\nsubg),
        \end{equation*}
        for an arbitrary $x\in\partial\fgfuchs$. One can easily check that the group homomorphism $\chi:G\rightarrow\fgfuchs$ is a homomorphic projection, where the SFT extended by $\chi$ is $\leftshift_G:\shiftsp_G\rightarrow\shiftsp_G$. Hence it is straightforward that $\accumset_x(\nsubg)=\pi(\accumset_x)$, $\escset_x(\nsubg)=\pi(\escset_x)$, $\unifescset_x(\nsubg)=\pi(\unifescset_x)$ and $\poincareexp(\nsubg)=\poincareexp(\identity_\fgfuchs)$. Combining this with \eqref{eqn:hausdim:udim}, we see that what we need prove is
        \begin{equation*}
            \dim_u(\accumset_x)= \hausdim(\escset_x)=\hausdim(\unifescset_x)=\poincareexp(\identity_\fgfuchs).
        \end{equation*}
        Therefore, we can directly apply Theorem \ref{mainthm:dimspec:symmetric} to deduce the equalities above, after checking that all the conditions in Theorem \ref{mainthm:dimspec:symmetric} are satisfied. We need to check the transitivity of $\chi^{-1}(\identity_\fgfuchs)$, the symmetry of the $\fgfuchs$-extension of $\leftshift_G:\shiftsp_G\rightarrow\shiftsp_G$ by $G$, and the symmetry of $u$.

        Since the rank of the free group $G$ is greater than the rank of the free group $\fgfuchs$, the surjective group homomorphism $\chi$ cannot be injective, implying that the kernel $\chi^{-1}(\identity_\fgfuchs)$ of $\chi$ has elements other than the identity $\identity_G$ of $G$. Take an arbitrary $\omega=\omega_0\cdots\omega_{|\omega|-1}\in \chi^{-1}(\identity_\fgfuchs)\subseteq G\setminus\set{\identity_G}$. Then, we can show that $\mathcal{I}:=\set{a\omega a^{-1}|a\in G_0\setminus\set{\omega_0^{-1},\omega_{\left|\omega\right|-1}}}$ is transitive. Indeed, for any two $a',a''\in\alphabet$, there is some $a\in G_0\setminus\set{\omega_0^{-1},\omega_{\left|\omega\right|-1},(a')^{-1},a''}$ because $\# G_0\geq 6$, and we clearly have $a'a\omega a^{-1}a''$ is reduced, which shows the transitivity of $\mathcal{I}$. In addition, note that for any $a\in G_0\setminus\set{\omega_0^{-1},\omega_{\left|\omega\right|-1}}$, we have $\chi(a\omega a^{-1})=\chi(a)\chi(\omega)\chi(a)^{-1}=\chi(a)\identity_G\chi(a)^{-1}=\identity_G$. Therefore, $\mathcal{I}\subseteq\chi^{-1}(\identity_\fgfuchs)$. Hence we can deduce the transitivity of $\chi^{-1}(\identity_\fgfuchs)$.
        
        Define $\omega^\dagger=\omega^{-1}$ for all $\omega\in G$. It is clear that the $\fgfuchs$-extension of the SFT $f:\limitset(G)\rightarrow\limitset(G)$ is symmetric with respect to $\cdot^\dagger$. Thus, it only remains to check the symmetry of $u$ with respect to $\cdot^\dagger$. Using \eqref{eqn:birksum:hyperbmetric}, we have
        \begin{align*}
            \sup_{\omega\in G}|S_{\omega^\dagger} u-S_\omega u|&=\sup_{\omega\in G}|S_{\omega^{-1}} u-S_\omega u|\\
            &\leq2\Delta+\sup_{\omega\in G} |d_h(0,\omega 0)-d_h(0,\omega^{-1}0)|=2\Delta<+\infty,
        \end{align*}
        where the last equality holds because $\omega\in G$ is an isometry. The proof is thus complete.
    \end{proof}

\appendix
\section{Proof of Lemma \ref{lma:disjcylinders}} \label{app:lemma:proof}
    We prove Lemma \ref{lma:disjcylinders} using an argument similar to the argument in the proof of \cite[Lemma 3.1]{liu}.
    \begin{proof} [Proof of Lemma \ref{lma:disjcylinders}]
        Let $\mathcal{I}_0$ be a finite transitive subset of $\chi^{-1}(\identity_\fgfuchs)$, and let $L_0$ be the maximum of the lengths of the words in $\mathcal{I}_0$. For any $a,b$ in the alphabet $\alphabet$ and any $m\in\bbN$, we define
        \begin{equation*}
            \mathcal{C}^m(a,b):=\Set{\rho\in\chi^{-1}(\identity_\fgfuchs)| m\leq |\rho|\leq m+2L_0,\,a\rho b\in\shiftsp^*}.
        \end{equation*}
        For each $m\in\bbN$, let $\shiftsp^m:=\set{\omega\in\shiftsp^*| |\omega|=m}$, and we define a map $\theta^m_{a,b}:\chi^{-1}(\identity_\fgfuchs)\cap\shiftsp^m\rightarrow\mathcal{C}^m(a,b)$ as follows. For each $\omega\in\chi^{-1}(\identity_\fgfuchs)\cap\shiftsp^m$, we find $\rho,\rho'\in\mathcal{I}_0$ such that $a\rho\omega\rho'b\in\shiftsp^*$, and set $\theta^m_{a,b}(\omega)=\rho\omega\rho'\in\mathcal{C}^m(a,b)$.
        
        Observe that for any $\bar{\omega}\in\mathcal{C}^m(a,b)$, if there is some $\omega$ such that $\theta^m_{a,b}(\omega)=\bar{\omega}$, then $\omega$ must be a subword which begins with one of the first $L_0+1$ symbols in $\bar{\omega}$. This indicates that $\#(\theta^m_{a,b})^{-1}(\bar{\omega})\leq L_0+1$ for all $\bar{\omega}\in\mathcal{C}^m(a,b)$, which further gives
        \begin{equation} \label{eqn:lma:disjcylinders:proof1}
            \#\mathcal{C}^m(a,b)\geq \frac{\#(\chi^{-1}(\identity_\fgfuchs)\cap\shiftsp^m)}{L_0+1}.
        \end{equation}
        As we have seen in Proposition \ref{prop:delta:+} that $\poincareexp(\identity_\fgfuchs)>0$, we take an arbitrary $p\in(0,\poincareexp(\identity_\fgfuchs))$, and we have
        \begin{equation*}
            \sum_{m=1}^\infty \sum_{\omega\in\chi^{-1}(\identity_\fgfuchs)\cap\shiftsp^m}\exp(-pS_\omega u)=\sum_{\omega\in\chi^{-1}(\identity_\fgfuchs)}\exp(-pS_\omega u)=+\infty.
        \end{equation*}
        It follows that
        \begin{equation*}
            \limsup_{m\rightarrow+\infty}\frac{1}{m}\log\sum_{\omega\in\chi^{-1}(\identity_\fgfuchs)\cap\shiftsp^m}\exp(-pS_\omega u)\geq 0.
        \end{equation*}
        Since $\log\sum_{\omega\in\chi^{-1}(\identity_\fgfuchs)\cap\shiftsp^m}\exp(-pS_\omega u)\leq -mp\min u+\log\#(\chi^{-1}(\identity_\fgfuchs)\cap\shiftsp^m)$, we have
        \begin{equation*}
            \limsup_{m\rightarrow+\infty}\frac{1}{m}\log\#(\chi^{-1}(\identity_\fgfuchs)\cap\shiftsp^m)\geq p\min u>0.
        \end{equation*}
        Combining this with \eqref{eqn:lma:disjcylinders:proof1}, we see that by taking a sufficiently large $m$, which does not depend on $a,b\in\alphabet$, we can let
        \begin{equation*}
            \#\mathcal{C}^m(a,b) > N:=(\#\alphabet^2-1)(2L_0+1)+(\#\alphabet+1)(\#\alphabet^{2L_0+1}-1)
        \end{equation*}
        for all $a,b\in\alphabet$.

        Label the elements of $\alphabet$ such that $\alphabet=\set{a_1,\cdots,a_{\#A}}$, and endow $\alphabet\times \alphabet$ with the lexicographic ordering. We choose $\rho(a,b)\in\mathcal{C}^m(a,b)$ for each $(a,b)\in\alphabet\times\alphabet$ following this ordering. The way we choose $\rho(a,b)$ goes as follows. At the first step, we pick an arbitrary $\rho(a_1,a_1)$ from $\mathcal{C}^m(a_1,a_1)$. At the $k$-th step with $k\in\set{2,\cdots,\#\alphabet^2}$, we choose $\rho(a_i,a_j)\in\mathcal{C}^m(a_i,a_j)$ for some $(a_i,a_j)\in\alphabet\times\alphabet$ satisfying that for every $\rho(a_{i'},a_{j'})$ which has been defined, $\rho(a_i,a_j)$ is not a prefix of $\rho(a_{i'},a_{j'})$ and $\rho(a_{i'},a_{j'})$ is not a prefix of $\rho(a_i,a_j)$ either. Such a choice of $\rho(a_i,a_j)$ is possible. To see this, note that for any $\rho\in\bigcup_{a,b\in\alphabet}\mathcal{C}^m(a,b)$, whose length is clearly between $m$ and $m+2L_0$, we have
        \begin{align*}
            \#\Set{\rho'\in\bigcup_{k=m}^{m+2L_0}\shiftsp^k|\rho'\text{ is a prefix of }\rho}&\leq 2L_0+1;\\
            \#\Set{\rho'\in\bigcup_{k=m}^{m+2L_0}\shiftsp^k|\rho\text{ is a prefix of }\rho'}&\leq\sum_{l=0}^{2L_0}\#\alphabet^l=\frac{\#\alphabet^{2L_0+1}-1}{\#\alphabet-1}.
        \end{align*}
        When we choose $\rho(a_i,a_j)$ at step $k$, the number of words which have been chosen in the previous steps is $k-1\leq\#\alphabet^2-1$. Therefore, the total number of words we need to avoid when picking $\rho(a_i,a_j)$ is less than or equal to $N$. As we picked $m$ large enough so that $\#\mathcal{C}^m(a_i,b_j)> N$, a word $\rho(a_i,a_j)\in\mathcal{C}^m(a_i,a_j)$ satisfying our requirements must exist.
        
        After choosing $\rho(a,b)$ for all $a,b\in\alphabet$ following the way we described in the previous paragraph, we set $\mathcal{I}:=\set{\rho(a,b)|a,b\in\alphabet}$. Evidently, $\mathcal{I}$ is finite and transitive. Moreover, the requirement we imposed on $\rho(a_i,a_j)$ previously guarantees that for any two distinct $(a,b),(a',b')\in\alphabet\times\alphabet$, $\rho(a,b)$ is not a prefix of $\rho(a',b')$ and $\rho(a',b')$ is not a prefix of $\rho(a,b)$, which implies that $[\rho(a,b)]\cap[\rho(a',b')]\neq\varnothing$. This completes our proof.
    \end{proof}


\end{document}